\documentclass[11pt]{amsart}

\usepackage{amscd}
\usepackage{amsmath, amssymb}
\usepackage{amsfonts}
\newcommand{\de}{\partial}

\newcommand{\ov}[1]{\overline{#1}}

\newcommand{\vol}{\mathrm{Vol}}

\newcommand{\ve}{\varepsilon}

\renewcommand{\leq}{\leqslant}
\renewcommand{\geq}{\geqslant}

\begin{document}
\newtheorem{claim}{Claim}
\newtheorem{theorem}{Theorem}[section]
\newtheorem{conjecture}[theorem]{Conjecture}
\newtheorem{lemma}[theorem]{Lemma}
\newtheorem{corollary}[theorem]{Corollary}
\newtheorem{proposition}[theorem]{Proposition}
\newtheorem{question}[theorem]{question}
\newtheorem{conj}[theorem]{Conjecture}
\newtheorem{defn}[theorem]{Definition}
\theoremstyle{definition}
\newtheorem{rmk}[theorem]{Remark}

\numberwithin{equation}{section}

\newenvironment{example}[1][Example]{\addtocounter{remark}{1} \begin{trivlist}
\item[\hskip
\labelsep {\bfseries #1  \thesection.\theremark}]}{\end{trivlist}}
\title{Regularity of the volume function}

\author{Junyu Cao}
\author{Valentino Tosatti}
\address{Courant Institute of Mathematical Sciences, New York University, 251 Mercer St, New York, NY 10012}
\email{junyu.cao@nyu.edu}
\email{tosatti@cims.nyu.edu}
%\thanks{MSC 2020: 14C20}

\begin{abstract}
We prove the optimal $C^{1,1}$ regularity of the volume function on the big cone of a projective manifold, and investigate its regularity when restricted to segments moving in ample directions.
\end{abstract}

\maketitle

\section{Introduction}
\subsection{The volume function}
Given a Cartier divisor $D$ on a smooth complex projective variety $X^n$, its volume is defined as
\begin{equation}
\vol(D):=\limsup_{m\to+\infty}\frac{n!\cdot h^0(X,mD)}{m^n}\in\mathbb{R}_{\geq 0}.
\end{equation}
The volume of a divisor is a basic and yet delicate object, which has been the focus of much recent work (see e.g. \cite{Bo,BFJ,ELMNP3,Laz,LM,WN}) and has been pivotal in several recent developments in birational geometry, such as \cite{HM,HMX,Tak}.

In \cite[\S 2.2.C]{Laz}, Lazarsfeld established the basic properties of the volume, including the fact that $\vol(D)$ only depends on the numerical class of $D$. We can thus define a volume function
\begin{equation}
\vol:N^1(X)\to\mathbb{R}_{\geq 0},
\end{equation}
on the N\'eron-Severi group $N^1(X)$. This function is homogeneous of degree $n$, thus it can be naturally extended to $N^1(X)\otimes_{\mathbb{Z}}\mathbb{Q}$, and this extension is locally Lipschitz by \cite[Theorem 2.2.44]{Laz}, hence it can be further extended to a locally Lipschitz function
\begin{equation}
\vol:N^1(X,\mathbb{R})\to\mathbb{R}_{\geq 0},
\end{equation}
defined on the real N\'eron-Severi group $N^1(X,\mathbb{R})=N^1(X)\otimes_{\mathbb{Z}}\mathbb{R}$. In \cite{Bo}, Boucksom gave a transcendental formula for the volume of a divisor, which allowed him to extend the definition of the volume function to any compact K\"ahler manifold $X$ (or even just a compact complex manifold bimeromorphic to a K\"ahler one). This way, he obtained a continuous function
\begin{equation}
\vol:H^{1,1}(X,\mathbb{R})\to\mathbb{R}_{\geq 0},
\end{equation}
where of course $N^1(X,\mathbb{R})\subset H^{1,1}(X,\mathbb{R})$. When $\alpha\in H^{1,1}(X,\mathbb{R})$ is nef, we have the simple formula
$$\vol(\alpha)=\int_X\alpha^n.$$

\subsection{Optimal regularity}
In this note we investigate the question of the optimal regularity of the volume function. Let $X^n$ be a compact K\"ahler manifold and let $\mathcal{B}_X\subset H^{1,1}(X,\mathbb{R})$ be the open cone of big $(1,1)$-classes, which are precisely the $(1,1)$-classes $\alpha$ with $\vol(\alpha)>0$. On the big cone the function $\vol^{\frac{1}{n}}$ is concave, hence $\vol$ is locally Lipschitz continuous.

Much more is true: in \cite{BFJ,LM}, Boucksom-Favre-Jonsson and Lazarsfeld-Musta\c{t}\u{a} independently proved that when $X$ is projective $\vol$ is $C^1$ differentiable inside $\mathcal{B}_X\cap N^1(X,\mathbb{R})$, and this result was extended to all of $\mathcal{B}_X$ by Witt Nystr\"om \cite{WN}, and even to the case of $X$ K\"ahler provided the BDPP Conjecture holds \cite[Conj. 2.3, Conj. 10.1(ii)]{BDPP} (the main result of \cite{WN} is that the BDPP Conjecture holds when $X$ is projective). These works also gave an attractive expression for the gradient of $\vol$: on $X$ projective (or compact K\"ahler satisfying BDPP), given $\alpha\in\mathcal{B}_X$ and $\beta\in H^{1,1}(X,\mathbb{R})$ we have
\begin{equation}\label{diff}
\frac{d}{dt}\bigg|_{t=0}\vol(\alpha+t\beta)=n\beta\cdot\langle\alpha^{n-1}\rangle.\end{equation}
On the other hand, in general $\vol$ cannot be of regularity higher than $C^{1,1}_{\rm loc}(\mathcal{B}_X)$, since when $X$ is the blowup of $\mathbb{P}^2$ at one point, the volume function can be easily computed explicitly \cite[Example 2.2.46]{Laz} and it is of class $C^{1,1}$ but not $C^2$ inside $\mathcal{B}_X$ (its Hessian is a bounded but discontinuous function).

In \cite[Question 1.4.4]{FLT}, it was recently asked what the optimal regularity of $\vol$ is inside $\mathcal{B}_X$, and whether it could be $C^{1,1}$ in general. Our first result answers this question affirmatively:

\begin{theorem}\label{t1}
Let $X$ be a projective manifold (or more generally a compact K\"ahler manifold satisfying the BDPP Conjecture). Then $\vol\in C^{1,1}_{\rm loc}(\mathcal{B}_X)$.
\end{theorem}

One can also wonder about the regularity of $\vol$ at points in $\de\mathcal{B}_X$, or equivalently on all of $H^{1,1}(X,\mathbb{R})$ (since $\vol$ vanishes identically outside $\mathcal{B}_X$). In \cite[Theorem 2.2.44]{Laz}, Lazarsfeld showed that when $X$ is projective, $\vol$ is locally Lipschitz on all of $N^1(X,\mathbb{R})$, which is again optimal in general since $\vol$ need not be differentiable at points of $\de\mathcal{B}_X$ (e.g. again in the case of the blowup of $\mathbb{P}^2$ at one point). Our second result extends this to all of $H^{1,1}(X,\mathbb{R})$:

\begin{theorem}\label{t2}
Let $X$ be a projective manifold (or more generally a compact K\"ahler manifold satisfying the BDPP Conjecture). Then $\vol\in \mathrm{Lip}_{\rm loc}(H^{1,1}(X,\mathbb{R}))$.
\end{theorem}
\begin{rmk}It would be interesting to prove this result for all compact K\"ahler manifolds without appealing to the BDPP Conjecture.\end{rmk}

There has also been much interest recently in the optimal regularity of the restriction of $\vol$ to segments of the form $\alpha+t\omega$ where $\alpha\in\ov{\mathcal{B}_X}$ and $\omega$ is a K\"ahler class, see e.g. \cite{Leh, McC, Ec, Les, FLT}. In particular, in \cite{FLT} it is proved that if $X$ is projective (or compact K\"ahler satisfying BDPP) and $\alpha\in\de\mathcal{B}_X$, then \begin{equation}\label{voll}
t\mapsto \vol(\alpha+t\omega),\quad t\in [0,1],
\end{equation}
is always $C^1$ differentiable on $[0,1]$, but there are examples where it is not $C^{1,\gamma}$ on $[0,\ve)$ for any $\gamma,\ve>0$. Our next result shows that the same holds ``with one more derivative'' when $\alpha\in\mathcal{B}_X$:

\begin{theorem}\label{irr}
Let $X$ be a projective manifold, with $\alpha\in\mathcal{B}_X$ and $\omega$ K\"ahler. Then the function in \eqref{voll} is $C^{1,1}$ differentiable on $[0,1]$ but there are examples where it is not $C^{2,\gamma}$ on $[0,\ve)$ for any $\gamma,\ve>0$.
\end{theorem}

\begin{rmk}\label{meno}
In contrast to the result in Theorem \ref{irr}, Rob Lazarsfeld conjectures that given $\alpha\in\mathcal{B}_X$ and $\omega$ K\"ahler, the function
\begin{equation}\label{voll2}
t\mapsto \vol(\alpha-t\omega),\quad t\in [0,1],
\end{equation}
should be smooth (or even real analytic) on $[0,\ve)$ for some small $\ve>0$. In the examples in Theorem \ref{irr}, the regularity of \eqref{voll2} is unclear, and it may well be smooth on $[0,\ve)$.
\end{rmk}

\begin{rmk}
In \cite[Conjecture 2.18]{ELMNP3} and \cite[Remark 2.2.52]{Laz}, it is conjectured that roughly speaking there should exists a wall-chamber decomposition of $\mathcal{B}_X$ (with walls allowed to accumulate) such that $\vol$ is $C^\infty$ (or even real analytic) on each open chamber. This is true on surfaces \cite{BKS} and on toric manifolds \cite{HKP}. Our proof of Theorem \ref{t1} shows that $\vol$ is $3$ times differentiable Lebesgue a.e. on $\mathcal{B}_X$, see Remark \ref{thrice}. It would be very interesting to improve this to infinite differentiability Lebesgue a.e. on $\mathcal{B}_X$.
\end{rmk}

\begin{rmk}
In \cite{LM}, Lazarsfeld-Musta\c{t}\u{a} associate a volume function to an arbitrary multi-graded linear series, which enjoys many of the same formal properties as $\vol$.
In \cite{KLM}, K\"uronya-Lozovanu-Maclean showed that every continuous, $n$-homogeneous, log-concave function can arise as the volume function of some multi-graded linear series, and in particular there are examples of such volume functions which are nowhere $3$ times differentiable. Indeed, our results rely crucially on the $C^1$ differentiability of the volume function and on the formula \eqref{diff} for its gradient, and thus no analogous property can hold for general multi-graded linear series.
\end{rmk}
The paper is organized as follows. In section \ref{s2} we prove Theorems \ref{t1} and \ref{t2}. We give two different proofs of Theorem \ref{t1}, and both rely crucially on the formula \eqref{diff} for the gradient of $\vol$. In section \ref{s3} we give the proof of Theorem \ref{irr}, by taking a $\mathbb{P}^1$-bundle over the Calabi-Yau $3$-fold considered in \cite{FLT}.

\subsection*{Acknowledgments} We are grateful to Guido De Philippis, Simion Filip, Mattias Jonsson, Rob Lazarsfeld, John Lesieutre and Jian Xiao for many useful discussions, and to the referee for very useful comments. The second-named author was partially supported by NSF grant DMS-2404599.

\section{Optimal $C^{1,1}$ regularity}\label{s2}
In this section we give two proofs of Theorem \ref{t1}. For the first proof we will use the following result, which is well-known to the experts, cf. \cite[Theorem 2.16]{BFJ} in the projective case:
\begin{proposition}\label{conv}
Let $(X^n,\omega)$ be a compact K\"ahler manifold.
The function $\mathcal{B}_X\to\mathbb{R}_{>0}$ given by
$$\alpha\mapsto\left(\omega\cdot\langle\alpha^{n-1}\rangle\right)^{\frac{1}{n-1}}$$
is concave.
\end{proposition}
\begin{proof}
Since $\omega$ is K\"ahler, \cite[Lemma 3.1]{Xi} (see also \cite[p.219]{BEGZ}) gives us that
$$\omega\cdot\langle\alpha^{n-1}\rangle=\langle\alpha^{n-1}\cdot\omega\rangle.$$
Thus, given $\alpha,\beta\in\mathcal{B}_X$ and $t\in [0,1]$, we want to show that
\begin{equation}\label{gln}
\langle(t\alpha+(1-t)\beta)^{n-1}\cdot\omega\rangle\geq  \left(t\langle \alpha^{n-1}\cdot \omega\rangle^{\frac{1}{n-1}}+(1-t)\langle\beta^{n-1}\cdot \omega\rangle^{\frac{1}{n-1}}\right)^{n-1}.
\end{equation}
To prove this, we use same Fujita approximation method as in \cite[Theorem 5.4]{CT2} (taking $V=X$ there), applied to both classes $\alpha$ and $\beta$. Given $\ve>0$, this gives us two modifications
$\mu_{\alpha,\ve}:X_{\alpha,\ve}\to X$ and $\mu_{\beta,\ve}:X_{\beta,\ve}\to X$, where $X_{\alpha,\ve},X_{\beta,\ve}$ are compact K\"ahler manifolds, such that we can write
$$\mu_{\alpha,\ve}^*\alpha=A_{\alpha,\ve}+E_{\alpha,\ve},\quad \mu_{\beta,\ve}^*\beta=A_{\beta,\ve}+E_{\beta,\ve}$$
where $A_{\alpha,\ve},A_{\beta,\ve}$ are nef classes and $E_{\alpha,\ve},E_{\beta,\ve}$ are effective $\mathbb{R}$-divisors, such that
\begin{equation}\label{gl1}
\left|\langle\alpha^{n-1}\cdot\omega\rangle-(A_{\alpha,\ve}^{n-1}\cdot\mu_{\alpha,\ve}^*\omega)\right|\to 0,
\quad \left|\langle\beta^{n-1}\cdot\omega\rangle-(A_{\beta,\ve}^{n-1}\cdot\mu_{\beta,\ve}^*\omega)\right|\to 0,\quad \ve\downarrow 0.
\end{equation}
We can then take a modification $\mu_\ve:X_\ve\to X$ which dominates both $\mu_{\alpha,\ve}$ and $\mu_{\beta,\ve}$, pull back everything to $X_\ve$ without changing notation, so that
$\mu_\ve^*\alpha=A_{\alpha,\ve}+E_{\alpha,\ve}$ and so on. A basic property of positive products \cite[Thm 1.16]{BEGZ} then gives that
\begin{equation}\label{gl2}
\langle(t\alpha+(1-t)\beta)^{n-1}\cdot\omega\rangle=\langle(t\mu_\ve^*\alpha+(1-t)\mu_\ve^*\beta)^{n-1}\cdot\mu_\ve^*\omega\rangle\geq ((tA_{\alpha,\ve}+(1-t)A_{\beta,\ve})^{n-1}\cdot \mu_\ve^*\omega),
\end{equation}
and this can be expanded as
$$\left((tA_{\alpha,\ve}+(1-t)A_{\beta,\ve})^{n-1}\cdot \mu_\ve^*\omega\right)=\sum_{i=0}^{n-1}\binom{n-1}{i}t^i(1-t)^{n-i-1}(A_{\alpha,\ve}^i\cdot A_{\beta,\ve}^{n-i-1}\cdot \mu_\ve^*\omega),$$
and using here the Khovanskii-Teissier inequalities (cf. \cite[Proposition 5.2]{De}, \cite[Variant 1.6.2]{Laz}, \cite[Theorem 1.6.C1]{Gr})
$$(A_{\alpha,\ve}^i\cdot A_{\beta,\ve}^{n-i-1}\cdot \mu_\ve^*\omega)\geq (A_{\alpha,\ve}^{n-1}\cdot \mu_\ve^*\omega)^{\frac{i}{n-1}}(A_{\beta,\ve}^{n-1}\cdot \mu_\ve^*\omega)^{\frac{n-i-1}{n-1}},$$
we get
\[\begin{split}
&\left((tA_{\alpha,\ve}+(1-t)A_{\beta,\ve})^{n-1}\cdot \mu_\ve^*\omega\right)\\
&\geq \sum_{i=0}^{n-1}\binom{n-1}{i}t^i(1-t)^{n-i-1}(A_{\alpha,\ve}^{n-1}\cdot \mu_\ve^*\omega)^{\frac{i}{n-1}}(A_{\beta,\ve}^{n-1}\cdot \mu_\ve^*\omega)^{\frac{n-i-1}{n-1}}\\
&=\left(t(A_{\alpha,\ve}^{n-1}\cdot \mu_\ve^*\omega)^{\frac{1}{n-1}}+(1-t)((A_{\beta,\ve}^{n-1}\cdot \mu_\ve^*\omega)^{\frac{1}{n-1}}\right)^{n-1},
\end{split}\]
and combining this with \eqref{gl1} and \eqref{gl2}, and letting $\ve\downarrow 0$, proves the concavity statement in \eqref{gln}.
\end{proof}

\begin{proof}[First proof of Theorem \ref{t1}]
Fix a basis $\omega_1,\dots,\omega_N$ of $H^{1,1}(X,\mathbb{R})$ which consists of K\"ahler classes. Thanks to \eqref{diff}, the gradient of $\vol$ at a big class $\alpha\in\mathcal{B}_X$ is expressed with respect to this basis as
$$\nabla\vol\bigg|_{\alpha}=(n\omega_1\cdot\langle\alpha^{n-1}\rangle,\dots,n\omega_N\cdot\langle\alpha^{n-1}\rangle).$$
Thus, to show that $\vol$ is $C^{1,1}$ differentiable in $\mathcal{B}_X$, it suffices to show that given a K\"ahler class $\omega$ the function $\mathcal{B}_X\to\mathbb{R}_{>0}$ given by
\begin{equation}\label{f}
f(\alpha):=\omega\cdot \langle\alpha^{n-1}\rangle\end{equation}
is locally Lipschitz. Now, thanks to Proposition \ref{conv}, the (strictly positive) function $f^{\frac{1}{n-1}}$ is concave on $\mathcal{B}_X$. Since a concave function on an open subset of Euclidean space is locally Lipschitz, we get that $f^{\frac{1}{n-1}}$ is locally Lipschitz on $\mathcal{B}_X$, hence so is $f$ itself.
\end{proof}

\begin{rmk}\label{thrice}
Thanks to a classical theorem of Alexandrov, a concave function on an open subset of Euclidean space is twice differentiable Lebesgue a.e.
Thus, the function  $f^{\frac{1}{n-1}}$ is twice differentiable a.e. on $\mathcal{B}_X$, hence so is $f$ itself. This means that $\vol$ is three times differentiable a.e. on $\mathcal{B}_X$.
\end{rmk}

The second proof of Theorem \ref{t1} is more elementary and does not require the concavity property in Proposition \ref{conv}. We find it convenient to give an abstract formulation of the main ingredient used, so we let $V$ be a finite-dimensional real vector space, equipped with some Euclidean norm, with a nonempty open convex cone $\mathcal{B}\subset V$ and a function
\begin{equation}f:\mathcal{B}\to\mathbb{R}_{\geq 0}.\end{equation}

The following elementary lemma is based on the proof of \cite[Theorem 5.2 (a)]{ELMNP}, see also \cite[Proposition 3.1]{ELMNP2} for a related but different result.
\begin{proposition}\label{elem}
Suppose $f$ satisfies
\begin{itemize}
\item[(a)] $f$ is homogeneous of degree $d\geq 1$, i.e. $f(\lambda\alpha)=\lambda^d f(\alpha),$ for all $\alpha\in\mathcal{B}$ and $\lambda\in\mathbb{R}_{>0}$,
\item[(b)] $f$ is non-decreasing in the $\mathcal{B}$ directions, in the sense that for every $\alpha,\beta\in\mathcal{B}$ we have $f(\alpha+\beta)\geq f(\alpha)$,
\item[(c)] $f$ is locally bounded.
\end{itemize}
Then $f$ is locally Lipschitz in $\mathcal{B}$.
\end{proposition}
\begin{proof}
We may assume without loss of generality that $0\not\in\mathcal{B}$, since otherwise we would have $\mathcal{B}=V$ and then $f$ would be constant by property (b).
Let $\ell=\dim V$ and fix $\{e_i\}_{i=1}^\ell$ in $\mathcal{B}$ which form a basis of $V$. We may assume without loss of generality that the chosen norm is given by
\begin{equation}\left\|\sum_{i=1}^\ell c_i e_i\right\|=\sum_{i=1}^\ell|c_i|,\end{equation}
and we will use $B_r(\alpha)\subset V$ to denote the Euclidean ball of radius $r$ and center $\alpha\in\mathcal{B}$.
Given $\alpha_0\in\mathcal{B}$, choose $\ve>0$ sufficiently small so that $U:=B_\ve(\alpha_0)\subset \mathcal{B}$, and $B_\ve\left(\frac{\alpha_0}{2}\right)\subset \mathcal{B}$, and $\sup_U f<\infty$. In particular, given any $h\in\mathcal{B}$ we have
\begin{equation}
\frac{\alpha_0}{2}-\ve\frac{h}{\|h\|}\in  B_\ve\left(\frac{\alpha_0}{2}\right)\subset \mathcal{B},
\end{equation}
and hence
\begin{equation}\label{q1}
\frac{\alpha_0}{2}\frac{\|h\|}{\ve}-h\in \mathcal{B}.\end{equation}
Given any $\alpha\in U$ we have
\begin{equation}\alpha-\frac{\alpha_0}{2}=\frac{\alpha_0}{2}+(\alpha-\alpha_0)\in B_\ve\left(\frac{\alpha_0}{2}\right)\subset \mathcal{B},\end{equation}
hence
\begin{equation}\label{q2}
\alpha\frac{\|h\|}{\ve}-\frac{\alpha_0}{2}\frac{\|h\|}{\ve}\in\mathcal{B}.\end{equation}
If we assume that $h$ satisfies $\|h\|<\ve,$ then we clearly have
\begin{equation}\label{q3}
\alpha-\alpha\frac{\|h\|}{\ve}=\alpha\left(1-\frac{\|h\|}{\ve}\right)\in\mathcal{B},
\end{equation}
hence using \eqref{q1}, \eqref{q2} and \eqref{q3} we see that
\begin{equation}
\alpha-\frac{\alpha_0}{2}\frac{\|h\|}{\ve}=\left(\alpha-\alpha\frac{\|h\|}{\ve}\right)+\left(\alpha\frac{\|h\|}{\ve}-\frac{\alpha_0}{2}\frac{\|h\|}{\ve}\right)\in\mathcal{B},
\end{equation}
\begin{equation}
\alpha-h=\left(\alpha-\frac{\alpha_0}{2}\frac{\|h\|}{\ve}\right)+\left(\frac{\alpha_0}{2}\frac{\|h\|}{\ve}-h\right)\in\mathcal{B},
\end{equation}
so by property (b) we have
\begin{equation}\label{1}
f(\alpha)\geq f(\alpha-h)\geq f\left(\alpha-\frac{\alpha_0}{2}\frac{\|h\|}{\ve}\right)\geq f\left(\alpha-\alpha\frac{\|h\|}{\ve}\right),
\end{equation}
and property (a) gives
\begin{equation}\label{3}
f\left(\alpha-\alpha\frac{\|h\|}{\ve}\right)=\left(1-\frac{\|h\|}{\ve}\right)^d f(\alpha)\geq \left(1-d\frac{\|h\|}{\ve}\right)f(\alpha),
\end{equation}
so that combining \eqref{1} and \eqref{3} gives that for every $\alpha\in U$ and $h\in\mathcal{B}$ with $\|h\|<\ve$ we have
\begin{equation}\label{3a}
f(\alpha)\geq f(\alpha-h)\geq \left(1-d\frac{\|h\|}{\ve}\right)f(\alpha).
\end{equation}
Similarly, we have
\begin{equation}\label{4}\begin{split}
f(\alpha)&\leq f(\alpha+h)\leq f\left(\alpha+\frac{\alpha_0}{2}\frac{\|h\|}{\ve}\right)\leq f\left(\alpha+\alpha\frac{\|h\|}{\ve}\right)\\
&=\left(1+\frac{\|h\|}{\ve}\right)^d f(\alpha)\leq \left(1+2^d\frac{\|h\|}{\ve}\right)f(\alpha),
\end{split}
\end{equation}
and from \eqref{3a} and \eqref{4} we get
\begin{equation}\label{5}
|f(\alpha)-f(\alpha\pm h)|\leq \frac{C_d}{\ve}\|h\|f(\alpha),
\end{equation}
for every $\alpha\in U$ and $h\in\mathcal{B}$ with $\|h\|<\ve$. Given now any $\alpha_1,\alpha_2\in B_{\frac{\ve}{4}}(\alpha_0)\subset U$, write
\begin{equation}
\alpha_1-\alpha_2=\sum_{i=1}^\ell c_i e_i,
\end{equation}
so
\begin{equation}\label{6}
\sum_{i=1}^\ell |c_i|=\|\alpha_1-\alpha_2\|<\frac{\ve}{2}.
\end{equation}
Therefore, for any $1\leq j\leq \ell$, we have
\begin{equation}
\left\|\alpha_2+\sum_{i=1}^{j} c_ie_i-\alpha_0\right\|\leq \|\alpha_2-\alpha_0\|+\sum_{i=1}^{j}|c_i|<\frac{\ve}{4}+\frac{\ve}{2}<\ve,
\end{equation}
and so $\alpha_2+\sum_{i=1}^{j} c_ie_i\in U$. Using \eqref{5} and \eqref{6} we can then estimate
\begin{equation}\label{7}\begin{split}
|f(\alpha_1)-f(\alpha_2)|&=\left|f\left(\alpha_2+\sum_{i=1}^\ell c_ie_i\right)-f(\alpha_2)\right|\\
&\leq \sum_{j=1}^\ell \left|f\left(\alpha_2+\sum_{i=1}^j c_ie_i\right)-f\left(\alpha_2+\sum_{i=1}^{j-1} c_ie_i\right)\right|\\
&\leq \frac{C_d}{\ve}\sum_{j=1}^\ell |c_j|f\left(\alpha_2+\sum_{i=1}^{j-1} c_ie_i\right)\\
&\leq \frac{C_d}{\ve}\left(\sup_U f\right)\sum_{j=1}^\ell |c_j|\\
&= \frac{C_d}{\ve}\left(\sup_U f\right)\|\alpha_1-\alpha_2\|,
\end{split}\end{equation}
thus showing the desired Lipschitz bound for $f$ on $U$.
\end{proof}

\begin{proof}[Second proof of Theorem \ref{t1}]
We apply Proposition \ref{elem} with $V=H^{1,1}(X,\mathbb{R}),$ $\mathcal{B}=\mathcal{B}_X$. As explained in the discussion after \cite[Definition 1.17]{BEGZ}, the function $f:\mathcal{B}_X\to\mathbb{R}_{>0}$ defined in \eqref{f} satisfies properties (a), (b) and (c) of Proposition \ref{elem} because it is homogeneous of degree $n-1$, non-decreasing in the $\mathcal{B}_X$-directions, and continuous on $\mathcal{B}_X$. Proposition \ref{elem} thus shows that $f$ is locally Lipchitz in $\mathcal{B}_X$, and as explained earlier this shows that $\vol\in C^{1,1}_{\rm loc}(\mathcal{B}_X)$.
\end{proof}

To conclude this section, we give the proof of Theorem \ref{t2}. Observe that the locally Lipschitz property of $\vol$ at points in $\de\mathcal{B}_X$ is not obvious because while $\vol^{\frac{1}{n}}$ is concave on $\ov{\mathcal{B}_X}$, it is not concave on $H^{1,1}(X,\mathbb{R})$.
\begin{proof}[Proof of Theorem \ref{t2}]
Since $\vol$ is identically zero outside $\mathcal{B}_X$, and it is $C^{1,1}_{\rm loc}$ inside $\mathcal{B}_X$, it suffices to show that for every $\alpha\in \de\mathcal{B}_X$ and every $\beta\in H^{1,1}(X,\mathbb{R})$ close to $\alpha$ we have
\begin{equation}\label{todo}
\vol(\beta)=|\vol(\alpha)-\vol(\beta)|\leq C\|\alpha-\beta\|.
\end{equation}
Since \eqref{todo} is trivial when $\beta\not\in\mathcal{B}_X$, we assume that $\beta\in\mathcal{B}_X$. Then for every $\ve>0$ we have $\alpha+t(\beta-\alpha)\in\mathcal{B}_X, \ve\leq t\leq 1$. Now thanks to \cite{BFJ,LM,WN}, the volume is $C^1$ differentiable in $\mathcal{B}_X$ with
$$\left|\frac{d}{dt}\vol(\alpha+t(\beta-\alpha))\right|=n\left|\langle (\alpha+t(\beta-\alpha))^{n-1}\rangle\cdot (\beta-\alpha)\right|\leq n|\langle\beta^{n-1}\rangle \cdot (\beta-\alpha)|,$$
which is bounded above independent of $t\in [\ve,1]$ and also independent of $\ve>0$ small.
Hence, the compactness of $[\ve,1]$ implies that the Lipschitz constant of $\vol(\alpha+t(\beta-\alpha)), t\in [\ve,1],$ is uniformly bounded. Therefore,
$$|\vol(\alpha+\ve(\beta-\alpha))-\vol(\beta)|\leq C(1-\ve)\|\beta-\alpha\|,$$
with $C$ independent of $\ve$ and letting $\ve\to 0$ (using that $\vol$ is continuous on all of $H^{1,1}(X,\mathbb{R})$) we obtain \eqref{todo}.
\end{proof}

\section{Volume along segments}\label{s3}
In this section we prove Theorem \ref{irr}. Thanks to Theorem \ref{t1}, it suffices to prove:
\begin{proposition}\label{exo}
There exist a projective $4$-fold $X$ with a big class $\alpha\in N^1(X, \mathbb{R})$ and an ample class $\omega\in N^1(X)$ such that the function
$$t\mapsto \vol(\alpha+t\omega)$$
is not $C^{2,\gamma}$ on $[0,\ve)$ for any $\gamma,\ve>0$.
\end{proposition}

\begin{proof}
Let $B$ be a general hypersurface in $(\mathbb{P}^1)^4$ of degree $(2,2,2,2)$, equipped with an ample line bundle $A$, and let $D$ be the $\mathbb{R}$-divisor on $B$ constructed in \cite{FLT} with the properties that $\vol(D)=0$ and
$$t\mapsto \vol(D+tA)$$
is not $C^{1,\gamma}$ on $[0,\ve)$ for any $\gamma,\ve>0$.
Consider
the projective bundle $\pi:X=\mathbb{P}_{B}(\mathcal{O}_B\oplus A)\to B$. Let $\xi=\mathcal{O}_{\mathbb{P}_{B}(\mathcal{O}_B\oplus A)}(1)$ be the Serre line bundle over $X$. Then thanks to \cite[Lemma 2.3.2]{Laz} we have that
$$\omega:=\xi+\pi^*A,$$
is ample on $X$, and
$$\alpha:=\xi+\pi^*D,$$
is big. For clarity, we will write $\vol_X,\vol_B$ for the volume functions on $X$ and $B$ respectively.

The PhD thesis of Wolfe shows \cite[Theorem 6.3]{Wo} that for any $\mathbb{R}$-divisor $E$ on $B$ we have
\begin{equation}\label{wolfe}
\vol_X(\xi+\pi^*E)=n!\int_{0}^1\vol_B(E+sA)ds,
\end{equation}
For $t>0$ we then have
\[\begin{split}
\vol_X(\alpha+t\omega)&=\vol_X(\xi+\pi^*D +t(\xi+\pi^*A))\\
&=(1+t)^{n+1}\vol_X\left(\xi+\pi^*\left(\frac{D}{1+t}+\frac{t}{1+t}A\right)\right)\\
&=(1+t)^{n+1}n!\int_0^1 \vol_B\left(\frac{D}{1+t}+\left(s+\frac{t}{1+t}\right)A\right)ds\\
&=(1+t)n!\int_{0}^1\vol_B(D+(s(1+t)+t)A)ds,
\end{split}\]
and using the change of variable $u=s(1+t)+t$ we get
\[\begin{split}
\vol_X(\alpha+t\omega)&=n!\int_{t}^{1+2t}\vol_B(D+uA)du.
\end{split}\]
Thus, for $t>0$,
$$\frac{d}{dt}\vol_X(\alpha+t\omega)=2n!\vol_B(D+(1+2t)A)-n!\vol_B(D+tA),$$
and for $t$ close to zero the first piece $\vol_B(D+(1+2t)A)$ is a $C^{1,1}$ function by Theorem \ref{t1} (since $D+A$ is big), while the second piece $\vol_B(D+tA)$ fails to be $C^{1,\gamma}$ for any $\gamma>0$ by \cite{FLT}, and so $\frac{d}{dt}\vol_X(\alpha+t\omega)$ is not $C^{1,\gamma}$  on $[0,\ve)$ for any $\gamma,\ve>0$, as desired.
\end{proof}

\begin{rmk}
Following up on Remark \ref{meno}, let us see what this argument can tell us about the behavior of $\vol(\alpha-t\omega)$, $t\in [0,\ve)$. As above, using Wolfe's formula, we compute
\[\begin{split}
\vol_X(\alpha-t\omega)&=\vol_X(\xi+\pi^*D -t(\xi+\pi^*A))\\
&=(1-t)^{n+1}\vol_X\left(\xi+\pi^*\left(\frac{D}{1-t}-\frac{t}{1-t}A\right)\right)\\
&=(1-t)^{n+1}n!\int_0^1 \vol_B\left(\frac{D}{1-t}+\left(s-\frac{t}{1-t}\right)A\right)ds\\
&=(1-t)n!\int_{\frac{t}{1-t}}^1\vol_B(D+(s(1-t)-t)A)ds,
\end{split}\]
using that $\vol_B(D)=0$ hence $\vol_B\left(\frac{D}{1-t}+\left(s-\frac{t}{1-t}\right)A\right)=0$ for $s\leq \frac{t}{1-t}$. Using the change of variable $u=s(1-t)-t$ we get
\[\begin{split}
\vol_X(\alpha-t\omega)&=n!\int_{0}^{1-2t}\vol_B(D+uA)du.
\end{split}\]
Thus, for $t>0$,
\begin{equation}\label{rr}
\frac{d}{dt}\vol_X(\alpha-t\omega)=-2n!\vol_B(D+A-2tA),
\end{equation}
and since $D+A$ is big on $B$, the RHS of \eqref{rr} is at least $C^{1,1}$ for $t\in [0,\ve)$, hence in this example $\vol_X(\alpha-t\omega)$ is at least $C^{2,1}$. Furthermore, if we take $A$ sufficiently ample, then $D+A$ will also be ample, and hence in this case $\vol_X(\alpha-t\omega)$ is a polynomial for $t\in [0,\ve)$. This is of course in line with Lazarsfeld's conjectural statement in Remark \ref{meno}, and in stark contrast with Proposition \ref{exo}.
\end{rmk}

\end{document}